\documentclass[a4paper,12pt]{amsart}

\usepackage{amsmath}
\usepackage{amsthm}
\usepackage{amssymb}
\usepackage{stmaryrd}
\usepackage{MnSymbol}
\usepackage{bbm}
\usepackage{bm}
\usepackage[top=30truemm, bottom=30truemm, left=25truemm, right=25truemm]{geometry}

\numberwithin{equation}{section}

\newtheorem{theorem}{Theorem}[section]
\newtheorem{lemma}[theorem]{Lemma}

\newtheorem{proposition}[theorem]{Proposition}

\theoremstyle{definition}

\newtheorem*{acknowledgment}{Acknowledgment}

\renewcommand{\Re}{\mathrm{Re}\hspace{1pt}}
\renewcommand{\Im}{\mathrm{Im}\hspace{1pt}}

\newcommand{\meas}{\mathrm{meas}} 
\newcommand{\m}{\mathbf{m}}

\newcommand{\dist}{\mathrm{dist}}

\title[Discrete universality for Matsumoto zeta-functions]{Discrete universality for Matsumoto zeta-functions and nontrivial zeros of the Riemann zeta-function}
\author[K. Nakai]{Keita Nakai}
\date{}

\begin{document}

\begin{abstract}

In 2017, Garunk\v{s}tis, Laurin\v{c}ikas and Macaitien\.{e} proved the discrete universality theorem for the Riemann zeta-function sifted by imaginary part of nontrivial zeros of the Riemann zeta-function. 
This discrete universality has been extended in various zeta-functions and $L$-functions. 
In this paper, we generalize this discrete universality for Matsumoto zeta-functions.

\end{abstract}

\maketitle

\section{Introduction and the statement of main results}

Let $s = \sigma + it$ be a complex variable. 
The Riemann zeta-function $\zeta(s)$ is defined by the infinite series 
$\sum_{n=1}^{\infty} n^{-s}$ 
in the $\sigma > 1$, 
and can be continued meromorphically to the whole plane $\mathbb{C}$. 
In 1975, Voronin proved the following universality property of the Riemann zeta-function. 

\begin{theorem} [Voronin, \cite{Vo}]
Let $\mathcal{K}$ be a compact set in the strip $1/2 < \sigma < 1$ with connected complement, and let $f(s)$ be a non-vanishing continuous function on $\mathcal{K}$ that is analytic in the interior of $\mathcal{K}$. Then, for any $\varepsilon > 0$ 
\[
\liminf_{T \to \infty} \frac{1}{T} \meas \left\{\tau \in [0, T] :  \sup_{s \in \mathcal{K}} |\zeta(s + i\tau) -f(s)| < \varepsilon \right\} > 0,
\]
where $\meas$ denotes the 1-dimensional Lebesgue measure. 
\end{theorem}
In this universality, the shift $\tau$ can take arbitrary non-negative real values continuously. 
If the shift can take a certain values discretely and the universality holds by this shift, then we call it the discrete universality. 
First Reich~\cite{Re} proved the discrete universality on the Dedekind zeta-function, and many mathematicians  extended and generalized. 
See e.g. a survey paper \cite{Ma15} for the recent studies.

Let $0 < \gamma_1 \le \gamma_2 \le \dots$ be the imaginary parts of nontrivial zeros of the Riemann zeta-function. 
Montgomery~\cite{Mo73} conjectured the asymptotic relation 

\begin{align*}
\sum_{\substack{0 < \gamma, \gamma' \le T \\ 2\pi\alpha_1/\log{T} \le \gamma - \gamma' \le 2\pi\alpha_2/\log{T}}}1  \sim \left( \int_{\alpha_1}^{\alpha_2} \left(1 - \left( \frac{\sin{\pi{u}}}{\pi u}\right)^2 \right) \, du + \delta(\alpha_1, \alpha_2) \right) \frac{T}{2\pi} \log{T}
\end{align*}
as $T \to \infty$ for $\alpha_1 < \alpha_2$, where $\delta(\alpha_1, \alpha_2) = 1$ if $0 \in [\alpha_1, \alpha_2]$ and $\delta(\alpha_1, \alpha_2) = 0$ otherwise.
We consider the weak Montgomery conjecture: 
\begin{equation}
\sum_{\substack{0 < \gamma, \gamma' \le T \\ |\gamma - \gamma'| < c/\log{T}}}1 \ll T\log{T}
\label{weak}
\end{equation}
as $T \to \infty$ with a certain constant $c > 0$.
Under this conjecture, the following discrete universality for the Riemann zeta-function holds. 

\begin{theorem} [Garunk\v{s}tis, Laurin\v{c}ikas and Macaitien\.{e} \cite{GLM}]
Let $\mathcal{K}$ be a compact set in the strip $1/2 < \sigma < 1$ with connected complement, let $f(s)$ be a non-vanishing continuous function on $\mathcal{K}$ that is analytic in the interior of $\mathcal{K}$ and assume (\ref{weak}). Then, for any $\varepsilon > 0$ and $h > 0$, 
\[
\liminf_{N \to \infty} \frac{1}{N} \# \left\{1 \le k \le N : \sup_{s \in \mathcal{K}} |\zeta(s + ih\gamma_k) -f(s)| < \varepsilon \right\} > 0, 
\]
where $\#A$ denotes the cardinality of a set $A \subset \mathbb{N}$.
\end{theorem}

This universality theorem has been extended other zeta-functions and $L$-functions in \cite{La19}, \cite{LP}, \cite{La20}, \cite{BGKMP}, \cite{BFGMR}, \cite{Ka}. 
In this paper, we prove this universality for the class of Matsumoto zeta-functions. 

The notion of Matsumoto zeta-function $\varphi(s)$ is introduced by Matsumoto \cite{Ma90} and defined by 

\[
\varphi(s) = \prod_{n=1}^{\infty}\prod_{j=1}^{g(n)}(1 - a_n^{(j)}p_n^{-f(j, n)s})^{-1}, 
\]
where $g(n) \in \mathbb{N}$, $f(j, n) \in \mathbb{N}$, $a_n^{(j)} \in \mathbb{C}$, and $p_n$ is the $n$th prime number. 
Assuming the conditions 
\begin{equation}
g(n) \le C_1 p_n^{\alpha},\ |a_n^{(j)}| \le p_n^{\beta}
\label{prime}
\end{equation}
with nonnegative constants $\alpha$, $\beta$ and a positive constant $C_1$, we have 
\[
\varphi(s) = \sum_{n=1}^{\infty} \frac{b_n}{n^s}
\]
for $\sigma > \alpha + \beta + 1$. 
Furthermore $b_n \ll n^{\alpha + \beta + \varepsilon}$ for any $\varepsilon > 0$ 
if all prime factors of $n$ are large (see \cite[Appendix]{KM}). 

In this paper, we consider Matsumoto zeta-functions satisfying following assumptions. 

\begin{itemize}

\item[(i)] the condition (\ref{prime})

\item[(ii)] There exists $\alpha + \beta + 1/2 \le \rho < \alpha + \beta + 1$ such that the function $\varphi(s)$ is meromorphic in the half plane $\sigma \ge \rho$, all poles in this region are included in a compact set, and there is no pole on the line $\sigma = \rho$. 

\item[(iii)] There exists a positive constant $c_2$ such that 
\[
\varphi(\sigma +it ) \ll |t|^{c_2}
\]
as $|t| \to \infty$ for $\sigma > \rho$.

\item[(iv)] For $\rho \le \sigma < \min\{\Re(z) : \text{$z$ is a pole of $\varphi$} \}$, we have
\[
\int_{-T}^{T} |\varphi(\sigma + it)|^2\, dt \ll T. 
\]

\item[(v)] There exists a positive $\kappa$ such that 
\[
 \lim_{x \to \infty} \frac{1}{\pi(x)} \sum_{p_n \le x} \left|\sum_{\substack{j=1 \\ f(j, n) = 1}}^{g(n)} a_n^{(j)}\right|^2 p_n^{-2(\alpha + \beta)} = \kappa, 
\]
where $\pi(x)$ is the prime counting function.
\end{itemize}

Let $D_{\rho} = \{s \in \mathbb{C} : \rho < \sigma < \alpha + \beta + 1\}$.
Now we state the main theorem of this paper. 

\begin{theorem} \label{main}

Let $\mathcal{K}$ be a compact set in $D_{\rho}$ with connected complement, let $f(s)$ be a non-vanishing continuous function on $\mathcal{K}$ that is analytic in the interior of $\mathcal{K}$ and assume (\ref{weak}). Then, for any $\varepsilon > 0$ and $h > 0$, 
\[
\liminf_{N \to \infty} \frac{1}{N + 1} \# \left\{N \le k \le 2N : \sup_{s \in \mathcal{K}} |\varphi(s + ih\gamma_k) -f(s)| < \varepsilon \right\} > 0.
\]
\label{main theorem}
\end{theorem}

Sourmelidis, Srichan and Steuding \cite{SSS} proved similar universality for the Riemann zeta-function unconditionally. 
Their statement holds for the wider context of $\alpha$-points of $L$-functions from the Selberg class. 
However we have to take a subsequence of $\alpha$-points of $L$-functions from the Selberg class in their result. 
Using their results, we have following theorem without (\ref{weak}). 

\begin{theorem}
Let $\mathcal{K}$ and $f$ be same as Theorem~\ref{main}. 
Let $\mathcal{L}$ be a non-constant $L$-function in the Selberg class. 
Then there exists a subsequence of $\alpha$-points $(\rho_{\alpha, n_k})_{k \in \mathbb{N}}$ of $\mathcal{L}(s)$ such that for any $\varepsilon > 0$, 
\[
\liminf_{N \to \infty} \frac{1}{N + 1} \# \left\{N \le k \le 2N : \sup_{s \in \mathcal{K}} |\varphi(s + i\gamma_{\alpha, n_k}) -f(s)| < \varepsilon \right\} > 0
\]
holds, where $\gamma_{\alpha, n_k} = \Im(\rho_{\alpha, n_k})$. 
\end{theorem}

\section{Preliminaries}

We fix a compact subset $\mathcal{K}$ satisfying the assumptions of Theorem~\ref{main theorem}. 
We define $\rho< \sigma_0 <  \min_{s \in \mathcal{K}} \Re(s)$ as all poles are contained in $\sigma > \sigma_0$. 
Then we fix $\sigma_1,\ \sigma_2$ such that  
\[
\rho< \sigma_0 < \sigma_1 < \min_{s \in \mathcal{K}} \Re(s),\ \max_{s \in \mathcal{K}} \Re(s) < \sigma_{2} < \alpha + \beta + 1.
\]

Then, we define the rectangle region $\mathcal{R}$ by 
\begin{equation} \label{definition of R}
\mathcal{R} = (\sigma_1,\ \sigma_2) \times i \left( \min_{s \in \mathcal{K}} \Im(s) - 1/2,\ \max_{s \in \mathcal{K}}\Im(s) + 1/2 \right).
\end{equation}

Let $\mathcal{H}(\mathcal{R})$ be the set of all holomorphic functions on $\mathcal{R}$. 

We write $\mathcal{B}(T)$ for the Borel set of $T$ which is a topological space.  
Let $S^1 = \{ s \in \mathbb{C} : |s| = 1 \}$. 
For any prime $p$, we put $S_p = S^1$ and  $\Omega = \prod_p S_p.$
Then there exists the probability Haar measure $\m$ on $(\Omega, \mathcal{B}(\Omega))$. 
Then $\m$ is written by $\m = \otimes_p \m_p$, where $\m_p$ is the probability Haar measure on $(S_p, \mathcal{B}(S_p))$.
 
Let $\omega(p)$ be the projection of $\omega \in \Omega$ to the coordinate space $S_p$.
$\{\omega(p) : \text{$p$ prime}\}$ is a sequence of independent complex-valued random elements defined on the probability space $(\Omega, \mathcal{B}(\Omega), \m)$. 
For $\omega \in \Omega$, we put $\omega(1) :=1$, 
\[
\omega(n) := \prod_p \omega(p)^{\nu(n ; p)}, 
\]
where $\nu(n ; p)$ is the exponent of the prime $p$ in the prime factorization of $n$.
Here, we define the $\mathcal{H}(\mathcal{R})$-valued random elements
\begin{equation*}
\begin{split}
\varphi(s, \omega) &:= \prod_{n=1}^{\infty}\prod_{j=1}^{g(n)}(1 - a_m^{(j)}\omega(p)^{f(j, n)}p_n^{-f(j, n)s})^{-1}  \\
                     &= \sum_{n=1}^{\infty} \frac{b_n \omega(n)}{n^s}.
\end{split}
\end{equation*}

We define the probability measures on $(\mathcal{H}(\mathcal{R}), \mathcal{B}(\mathcal{H}(\mathcal{R})))$ by 
\begin{equation*}
P_N(A) = \frac{1}{N + 1} \# \left\{N \le k \le 2N : \varphi(s +ih\gamma_k) \in A \right\}, 
\end{equation*}
\begin{equation*}
P(A) = \m \left\{\omega \in \Omega : \varphi(s, \omega) \in A \right\}
\end{equation*} 
for $A \in \mathcal{B}(\mathcal{H}(\mathcal{R}))$. 

\section{A limit theorem}

This section is based on \cite{Ko}. 

\begin{lemma} \label{lemma smooth}
Let $\psi : [0, \infty) \to \mathbb{C}$ be smooth and assume that $\psi$ and all its derivatives decay faster than any polynomial at infinity, and let $\hat{\psi}(s) = \int_{0}^{\infty} \psi(x) x^{s-1}\, dx$ be the Mellin transform of $\psi$ on $\Re(s) > 0$.
\begin{enumerate}
\item[(1)] The Mellin transform $\hat{\psi}$ extends to a meromorphic function on $\Re(s) > -1$, with at most a simple pole at $s = 0$ with residue $\psi(0)$.
\item[(2)] For any real numbers $-1 < A < B$, the Mellin transform has rapid decay in the strip $A \le \sigma \le B$, in the sense that for any integer $k \ge 1$, there exists a constant $C = C(k, A, B) \ge 0$ such that 
\[
|\hat{\psi}(\sigma + it) | \le C (1 + |t|)^{-k}.
\]
for all $A \le \sigma \le B$ and $|t| \ge 1$.
\item[(3)] For any $\sigma > 0$ and any $x \ge 0$, we have the Mellin inversion formula 
\[
\psi(x) = \frac{1}{2 \pi i}\int_{\sigma - i \infty}^{\sigma + i \infty} \hat{\psi}(s)x^{-s}\, ds.
\]

\end{enumerate}
\end{lemma}
\begin{proof}
See \cite[Proposition~A.3.1]{Ko}.
\end{proof}

Now we fix a real-valued smooth function $\psi(x)$ on $[0, \infty)$ with compact support satisfying $\psi(x) = 1$ and $0 \le \psi(x) \le 1$. 
We note that we can take $\psi$ which satisfies the assumptions of Lemma~\ref{lemma smooth}. 
We put 
\[
\varphi_X(s) = \sum_{n=1}^{\infty}\frac{b_n \psi(n/X)}{n^{s}}
\]

\[
\varphi_X(s, \omega) = \sum_{n=1}^{\infty}\frac{b_n \omega(n) \psi(n/X)}{n^s}
\]
for $X \ge 2$.

\begin{lemma} \label{Prob1}
For all compact set $C \subset \mathcal{R}$
\[
\lim_{X \to \infty} \limsup_{N \to \infty} \frac{1}{N + 1}\sum_{k=N}^{2N} \sup_{s \in C} |\varphi(s + ih\gamma_k) - \varphi_X(s +ih\gamma_k)| = 0.
\]
\end{lemma}

\begin{proof}

From Lemma~\ref{lemma smooth} and definition of $\varphi_X(s)$, we see that 
\[
\varphi_X(s) = \frac{1}{2 \pi i} \int_{c - i\infty}^{c + i\infty} \varphi(s +w) \hat{\psi}(w) X^w\, dw
\]
for $c > \alpha + \beta +1$. 
We write $z_1,\dots, z_M$ for the poles of $\varphi$ contained in $\overline{D}_{\rho}$ and $r_1, \dots, r_M$ for its residues. 
Let $\delta(z)$ be a positive number satisfying $\Re(z) - \delta(z) = \sigma_0$ for $\Re(z) > \sigma_0$.
If $z \ne z_j$ for $1 \le j \le M$ and $\Re(z) > \sigma_0$, then by the residue theorem, we have 
\[
\varphi(z) - \varphi_X(z) = -\frac{1}{2 \pi i} \int_{-\delta(z) -i \infty}^{-\delta(z) + i\infty} \varphi(z +w) \hat{\psi}(w) X^w\, dw - \sum_{j=1}^{M} r_j \hat{\psi}(z_j -z)X^{z_j -z}.
\] 
Let $N$ be sufficiently large. 
Then $\varphi(s + ih\gamma_k) - \varphi_X(s + ih\gamma_k)$ is holomorphic on $\overline{\mathcal{R}}$ for $N \le k \le 2N$. 
Therefore we have  
\begin{align*}
&\sum_{k=N}^{2N} \sup_{s \in C}|\varphi(s + ih\gamma_k) - \varphi_X(s + ih\gamma_k)| \\
&= \frac{1}{2\pi} \sum_{k=N}^{2N} \sup_{s \in C} \left| \int_{\partial \mathcal{R}} \frac{\varphi(z + ih\gamma_k) - \varphi_X(z + ih\gamma_k)}{z - s}\, dz \right| \\
&\le \frac{1}{2\pi \dist(C, \partial\mathcal{R})} \int_{\partial \mathcal{R}} \sum_{k=N}^{2N} |\varphi(z + ih\gamma_k) - \varphi_X(z + ih\gamma_k)| |dz| \\
&\le \frac{1}{4\pi^2 \dist(C, \partial\mathcal{R})} \int_{\partial \mathcal{R}} \sum_{k=N}^{2N}  \int_{-\delta(z) -i \infty}^{-\delta(z) + i\infty} |\varphi(z + w + ih\gamma_k)| |\hat{\psi}(w) X^w|\, |dw| |dz| \\
&\ \  + \frac{|\partial \mathcal{R}|}{2\pi \dist(C, \partial\mathcal{R})} \sup_{z \in \partial \mathcal{R}}\sum_{k=N}^{2N} \sum_{j=1}^{M} |r_j| |\hat{\psi}(z_j -z -ih\gamma_k)| X^{\Re(z_j -z)} \\
&\le \frac{|\partial \mathcal{R}|}{4\pi^2 \dist(C, \partial\mathcal{R})} \sup_{z \in \partial \mathcal{R}} X^{-\delta(z)} \int_{-\infty}^{\infty} \sum_{k=N}^{2N} |\varphi(\Re(z) -\delta(z) + i\tau + ih\gamma_k)| |\hat{\psi}(-\delta + i\tau)| \, d\tau  \\
&\ \  + \frac{|\partial \mathcal{R}|}{2\pi \dist(C, \partial\mathcal{R})} \sup_{z \in \partial \mathcal{R}}\sum_{k=N}^{2N} \sum_{j=1}^{M} |r_j| |\hat{\psi}(z_j -z -ih\gamma_k)| X^{\Re(z_j -z)},
\end{align*}
where $\dist(C, \partial{\mathcal{R}})$ is the minimal distance between $C$ and $\partial{\mathcal{R}}$, and  $|\partial{\mathcal{R}}|$ is the length of $\partial{\mathcal{R}}$.  

We consider the first term. 
Using (\ref{weak}), assumption (iv) and the Gallagher Lemma on the discrete mean (see \cite[Lemma~1.4]{Mo71}), we have 

\[
\frac{1}{N + 1}\sum_{k=N}^{2N} |\varphi(\Re(z) -\delta(z) + i\tau + ih\gamma_k)| \ll 1 + |\tau|
\]
(cf. \cite[Lemma~2.7]{GLM}).
Thus, we obtain 
\begin{align*}
\frac{1}{N + 1}\sup_{z \in \partial \mathcal{R}} \int_{-\infty}^{\infty} \sum_{k=N}^{2N} |\varphi(\Re(z) -\delta(z) + i\tau + ih\gamma_k)| |\hat{\psi}(-\delta(z) + i\tau)| \, d\tau \ll 1.
\end{align*}

We consider the second term. 
It is known that 
\[
\gamma_k \sim \frac{2\pi k}{\log{k}}
\]
as $k \to \infty$, so we have  
\[
\gamma_k \gg \frac{k}{\log{k}}.
\]

By Lemma~\ref{lemma smooth} we have  
\begin{align*}
&\sup_{z \in \partial \mathcal{R}}\sum_{k=N}^{2N} \sum_{j=1}^{M} |r_j| |\hat{\psi}(z_j -z -ih\gamma_k)| X^{\Re(z_j -z)} \\
&\ll_{r_j} \sup_{z \in \partial \mathcal{R}}\sum_{j=1}^{M}X^{\Re(z_j -z)}  \sum_{k=N}^{2N} (1 + |\Im(z_j) -\Im(z) -h\gamma_k|)^{-1} \\
&\ll_{M, z_j, \mathcal{R}} \sup_{z \in \partial \mathcal{R}}\sup_{1 \le j \le M} X^{\Re(z_j -z)} \sum_{k=N}^{2N} \frac{\log{k}}{k} \\
&\ll X^{\frac{1}{2}} (\log{N}).
\end{align*}

Since $\delta(z) \ge \sigma_1 - \sigma_0 > 0$ for all $z \in \partial \mathcal{R}$, we conclude 
\begin{align*}
&\lim_{X \to \infty} \limsup_{N \to \infty} \frac{1}{N + 1} \sum_{k=N}^{2N} \sup_{s \in C} |\varphi(s + ih\gamma_k) - \varphi_X(s +ih\gamma_k)| \\
&\ll \lim_{X \to \infty} \limsup_{N \to \infty} (X^{-(\sigma_1 - \sigma_0)} + X^{\frac{1}{2}} (\log{N}) N^{-1}) = 0.
\end{align*}

\end{proof}

\begin{lemma} \label{mean}
The following statements hold.
\begin{enumerate}  
\item[(i)] 
The product $\prod_{n=1}^{\infty}\prod_{j=1}^{g(n)}(1 - a_m^{(j)}\omega(p)^{f(j, n)}p_n^{-f(j, n)s})^{-1}$ and the series $\sum_{n=1}^{\infty} b_n \omega(n)n^{-s}$ are holomorphic on the domain $\sigma > \alpha + \beta + 1/2$ for almost all $\omega \in \Omega$. 

\item[(ii)]
For $\sigma > (\sigma_0 + \sigma_1)/2$,
\[
\mathbb{E}^{\m}[|\varphi(s, \omega)|] \ll_{\mathcal{K}, \sigma_0}  1 + |t|
\]
holds.
\end{enumerate}
\end{lemma}

\begin{proof}

Applying the Kolmogorov theorem (see \cite[Theorem~1.2.11]{La96}) and the convergence theorem relating with orthogonal random elements (see \cite[Theorem~1.2.9]{La96}), we can prove (i).

We consider (ii).
Let 
\[
S(u) = \sum_{n \le u} \frac{b_n \omega(n)}{n^{\sigma_0}}.
\]
By the Cauchy--Schwartz inequality, there exists $M > 0$ such that 
\begin{align*}
\mathbb{E}^{\m}[|S_u|] &\le \left( \mathbb{E}^{\m}[|S_u|^2] \right)^{\frac{1}{2}} \\
                       &= \left(\sum_{n \le u} \frac{|b_n|^2}{n^{2\sigma_0}} \right)^{\frac{1}{2}} < M.
\end{align*}
By the definition of $S_u$, we have
\begin{align*}
\varphi(s, \omega) &= \int_{1^-}^{\infty} \frac{1}{u^{s - \sigma_0}}\, dS_u \\
&= (s - \sigma_0) \int_{1}^{\infty} \frac{S_u}{u^{s - \sigma_0 -1}}\, du.
\end{align*}
Thus, for $\Re(s) > (\sigma_0 + \sigma_1)/2$, we have 
\begin{align*}
\mathbb{E}^{\m}[|\varphi(s, \omega)|] &\le |s - \sigma_0| \int_{1}^{\infty} \frac{\mathbb{E}^{\m}[|S_u|]}{u^{\sigma - \sigma_0 - 1}}\, du \\
&\le M \frac{|s - \rho|}{\sigma - \sigma_0} \\
&\ll_{\mathcal{K}, \sigma_0} 1 + |t|.
\end{align*}

\end{proof}

From this lemma, we see that 
\begin{align*}
\varphi(s, \omega) &= \prod_{n=1}^{\infty}\prod_{j=1}^{g(n)}(1 - a_m^{(j)}\omega(p)^{f(j, n)}p_n^{-f(j, n)s})^{-1}\\
                   &= \sum_{n=1}^{\infty} \frac{b_n \omega(n)}{n^s}
\end{align*}
holds for $\sigma > \rho$ for almost all $\omega \in \Omega$ in the sense of analytic continuation.

\begin{lemma} \label{Prob2}
For all compact set $C \subset \mathcal{R}$,
\[
\lim_{X \to \infty} \mathbb{E}^{\m} [\sup_{s \in C}|\varphi(s, \omega) -\varphi_X(s, \omega)| ].
\]

\end{lemma}

\begin{proof}
By Lemma~\ref{mean}~(i) we have
\[
\sup_{s \in C}|\varphi(s, \omega) -\varphi_X(s, \omega)|
\ll X^{-\delta} \int_{\partial \mathcal{R}}\, |dz| \int_{-\infty}^{\infty}|\varphi(z - \delta +i\tau, \omega)| |\hat{\psi}(-\delta + i\tau)|\, d\tau, 
\]
where $\delta = (\sigma_1 - \sigma_0)/4$  in the same way as Lemma~\ref{Prob1}.
From Lemma~\ref{mean}~(ii), we have
\begin{align*}
\mathbb{E}^{\m} [\sup_{s \in C}|\varphi(s, \omega) -\varphi_X(s, \omega)| ]
&\ll X^{-\delta} \int_{\partial \mathcal{R}}\, |dz| \int_{-\infty}^{\infty}\mathbb{E}^{\m}[|\varphi(z - \delta +i\tau, \omega)|] |\hat{\psi}(-\delta + i\tau)|\, d\tau \\
&\ll X^{-\delta} \to 0
\end{align*}
as $X \to \infty$.

\end{proof}

We consider the discrete topology on $\mathbb{N}_{N \le n \le 2N} := \{n \in \mathbb{N}: N \le n \le 2N\}$. 
Then we define the probability measure on $(\mathbb{N}_{N \le n \le 2N}, \mathcal{B}(\mathbb{N}_{N \le n \le 2N}))$ by 
\[
\mathbb{P}_N(A) = \frac{1}{N + 1} \# A, 
\]
for $A \in \mathcal{B}(\mathbb{N}_{N \le n \le 2N})$. 
Furthermore, let $\mathcal{P}_0$ be a finite set of prime numbers, and we define the probability measure on 
$(\prod_{p \in \mathcal{P}_0} S_p, \mathcal{B}(\prod_{p \in \mathcal{P}_0} S_p))$ by 
\[
Q^{\mathcal{P}_0}_N(A) = \frac{1}{N  + 1} \# \left\{N \le k \le 2N : (p^{ih\gamma_k})_{p \in \mathcal{P}_0} \in A \right\}, 
\]
$A \in  \mathcal{B}(\prod_{p \in \mathcal{P}_0} S_p)$. 
Then, the next lemma holds. 

\begin{lemma} \label{Prob3}
The probability measure $Q^{\mathcal{P}_0}_N$ converges weakly to $\otimes_{p \in \mathcal{P}_0} \m_p$ as $N \to \infty$. 
\end{lemma}

\begin{proof}

We can prove this lemma in the same way as \cite[Theorem~2.3]{GLM}. 

\end{proof} 

\begin{proposition} \label{prop1}
The probability measure $P_N$ converges weakly to $P$ as $N \to \infty$.
\end{proposition}

\begin{proof}
Using the Portmanteau theorem (see \cite[Theorem~13.16]{Kl}), Lemma~\ref{Prob1}, Lemma~\ref{Prob2} and  Lemma~\ref{Prob3}, 
we can prove this Proposition (cf. \cite[Proposition~1]{En}). 

\end{proof}

\section{proof of main theorem}

Let 
\[
S := \{f \in \mathcal{H}(\mathcal{R}) :\text{ $f(s) \ne 0$ or $f(s) \equiv 0$} \}.
\]
Then, using assumption~(v) and same method \cite[Lemma~6]{La98}, we see that the support of the measure $P$ coincides with $S$ (cf. \cite[Lemma~6]{LM}).

\begin{proof}[Proof of Theorem~\ref{main theorem}]

Let $\mathcal{K}$ be a compact set in $D_\rho$ with connected complement, let $f$ be a non-vanishing continuous function on $\mathcal{K}$ that is analytic in the interior of $\mathcal{K}$. 
We define $\mathcal{R}$ as (\ref{definition of R}).
Fix $\varepsilon > 0$. 

By the Mergelyan theorem, there exists a polynomial $G(s)$ such that 
\[
\sup_{s \in \mathcal{K}} |f(s) - \exp(G(s))| < \varepsilon/2
\]
since $f$ is non-vanishing on $\mathcal{K}$. 
Here we define an open set of $\mathcal{H}(\mathcal{R})$ by 
\[\Phi(G):= \left\{g \in \mathcal{H}(\mathcal{R}) : \sup_{s \in \mathcal{K}} |g(s) - \exp(G(s))| < \varepsilon/2 \right\}.
\]
Applying the Portmanteau theorem, Proposition~\ref{prop1}, and a property of the support of $P$, 
we have  
\begin{align*}
&\liminf_{N \to \infty} \frac{1}{N + 1} \# \left\{N \le k \le 2N :  \sup_{s \in \mathcal{K}} |\varphi(s +ih\gamma_k) - \exp(G(s))| < \varepsilon/2 \right\} \\ \\
&= \liminf_{N \to \infty} P_N(\Phi(G)) \ge P(\Phi(G)) >0.
\end{align*}
Now, the inequality
\[
\sup_{s \in \mathcal{K}} |\varphi(s + ih\gamma_k) -f(s)| \le \sup_{s \in \mathcal{K}} |\varphi(s + ih\gamma_k) -\exp(G(s))| + \sup_{s \in \mathcal{K}} |\exp(G(s)) - f(s)|
\]
holds. 
Thus, we obtain the conclusion.

\end{proof}

\begin{acknowledgment}

The author would like to thank Professor Kohji Matsumoto for his helpful comments. 
The author is grateful to Professor J\"{o}rn Steuding for making me aware of theirs paper.
This work was financially supported by JST SPRING, Grant Number JPMJSP2125. 

\end{acknowledgment}

\begin{flushleft}
{\footnotesize
{\sc
Graduate School of Mathematics, Nagoya University, Chikusa-ku, Nagoya 464-8602, Japan.
}\\
{\it E-mail address}: {\tt m21029d@math.nagoya-u.ac.jp}
}
\end{flushleft}


\begin{thebibliography}{99}

\bibitem{BFGMR} A. Bal\v{c}i\={u}nas, V. Franckevi\v{c}, V. Garbaliauskien\.{e}, R. Macaitien\.{e} and A. Rimkevi\v{c}ien\.{e}, Universality of zeta-functions of cusp forms and non-trivial zeros of the Riemann zeta-function, Math. Model. Anal. \textbf{26}(1), 82–93, 2021. 

\bibitem{BGKMP} A. Bal\v{c}i\={u}nas, V. Garbaliauskien\.{e}, J. Karali\={u}nait\.{e}, R. Macaitien\.{e}, J. Petu\v{s}kinait\.{e} and A. Rimkevi\v{c}ien\.{e}, Joint discrete approximation of a pair of analytic functions by periodic zeta-functions, Math. Model. Anal. \textbf{25}, no.1, 71–87, 2020.

\bibitem{Bi} P. Billingsley, \emph{Convergence of Probability Measures}, 2nd ed, Wiley, NewYork,1999.

\bibitem{Co} J. B. Conway, \emph{Functions of One Complex Variabe I}, 2nd ed, Springer, 1978. 

\bibitem{En} K. Endo, Universality theorem for the iterated integrals of the logarithm of the Riemann zeta-function, Lith. Math. J. \textbf{62} (2022), 315--332. 

\bibitem{GLM} R. Garunk\v{s}tis, A. Laurin\v{c}ikas, and  R. Macaitien\.{e}, Zeros of the Riemann zeta-function and its universality, Acta Arith., \textbf{181}, no. 2, 127--142, 2017.

\bibitem{Ka} R. Ka\v{c}inskait\.{e}, On discrete universality in the Selberg--Steuding class, Sib. Math. J. \textbf{63}(2), 277--285, 2022.

\bibitem{KM} R. Ka\v{c}inskait\.{e} and K. Matsumoto, Remarks on the mixed joint universality for a class of zeta-functions, Bull. Aust. Math. Soc., \textbf{98}(2):187--198, 2017.

\bibitem{Kl} A. Klenke, \emph{Probability Theory: A Comprehensive Course}, 3rd ed, Springer Science \& Business Media, 2020.

\bibitem{Ko} E. Kowalski, \emph{An Introduction to Probabilistic Number Therory}, Cambridge University Press, 2021. 

\bibitem{La96} A. Laurin\v{c}ikas, \emph{Limit Theorems for the Riemann Zeta-function}, Kluwer, 1996. 

\bibitem{La98} A. Laurin\v{c}ikas, On the Matsumoto zeta-function, Acta Arith., \textbf{84}, 1--16, 1998.

\bibitem{La19} A. Laurin\v{c}ikas, Non-trivial zeros of the Riemann zeta-function and joint universality theorems, J. Math. Anal. Appl., \textbf{475}(1), 385--402, 2019. 

\bibitem{La20} A. Laurin\v{c}kas, Zeros of the Riemann zeta-function in the discrete universality of the Hurwitz zeta-function, Stud. Sci. Math. Hung., \textbf{57}, no. 2, 147--164, 2020. 

\bibitem{LM} A. Laurin\v{c}ikas and K. Matsumoto, The universality of zeta-functions attached to certain cusp forms, Acta Arith., \textbf{98} (2001), 345-359.

\bibitem{LP} A. Laurin\v{c}ikas and J. Petu\v{s}kinait\.{e}, Universality of Dirichlet $L$-functions and non-trivial zeros of the Riemann zeta-function. Sb. Math., \textbf{210}(12), 1753--1773, 2019.

\bibitem{Ma90} K. Matsumoto, Value-distribution of zeta-functions, in: Lecture Notes in Math. 1434, Springer, 1990, 178--187.

\bibitem{Ma15} K. Matsumoto, A survey on the theory of universality for zeta and $L$-functions, Number theory, 95--144, Ser. Number Theory Appl, 11, World Sci. Publ., Hackensack, NJ, 2015.

\bibitem{Mo71} H. L. Montgomery, \emph{Topics in Multiplicative Number Theory}, Lecture Notes in Math. 227, Springer, Berlin, 1971. 

\bibitem{Mo73} H. L. Montgomery, The pair correlation of zeros of the zeta function, in: Analytic Number Theory (St. Louis, MO, 1972), H. G. Diamond (ed.), Proc. Sympos. Pure Math. \textbf{24}, Amer. Math. Soc., Providence, RI, 1973, 181--193.

\bibitem{Re} A. Reich, Werteverteilung von Zetafunktionen, Arch. Math. (Basel) \textbf{34} (1980), 440--451. 

\bibitem{SSS} A. Sourmelidis, T. Srichan and J. Steuding, On the vertical distribution of values of $L$-functions in the Selberg class, Int. J. Number Theory, \textbf{18}, No. 02, 277--302, 2022.

\bibitem{Vo} S. M. Voronin, Theorem on the universality of the Riemann zeta-function, Izv. Akad. Nauk SSSR Ser. Mat. \textbf{39} (1975) 475–486 (in Russian); Math. USSR Izv. \textbf{9} (1975), 443–453.
\end{thebibliography}
\end{document}